\documentclass{rendiconti}

\RendicontiPagina{1}{Mattia Scomparin}{Mattia Scomparin}

\usepackage{amsmath}
\usepackage{amsfonts}
\usepackage{slashed}
\usepackage{xcolor} 
\newcommand{\vp}{\varphi}
\newcommand{\dx}{\dot{x}}
\newcommand{\ddx}{\ddot{x}}
\newcommand{\ve}{\varepsilon}
\newcommand{\eb}{e^{\pm\!\int_t\!b d\tau}}
\newcommand{\ebb}{e^{\pm\!\int_t\!b_\pm d\tau}}
\newcommand{\B}{\bar{B}}
\newcommand{\dt}{\boldsymbol{\cdot}}
\def\smallint{\begingroup\textstyle \int\endgroup}
\newcommand{\p}[1] {\partial_{#1}}

\title{New nonlocal constants and first integrals for nonlinear Jacobi-type equations}
\author{Mattia Scomparin}
\dedicated{mattia.scompa@gmail.com}

\summary{
We prove a set of general theorems that provide new nonlocal constants and first integrals for nonlinear Jacobi-type ordinary differential equations. Applications include equations of the Painlev\'{e}-Gambier classification.
}

\intesta

\begin{document}
\maketitle


\section{Introduction}
\label{sec:intro}

We consider Lagrangians $L=L(t,x,\dx)$ whose \textit{Euler-Lagrange} equations of motion are Ordinary Differential Equations (ODEs) like this
\begin{equation}\label{eq:E-L}
\big(\partial_{\dx} L(t,x,\dx)\big)^{\dt}-\partial_{x} L(t,x,\dx)=0\,,
\end{equation}
with $x=x(t)$ and $\dx\equiv dx/dt$. The $\partial$ symbol stands for partial derivative, e.g. $\partial_{\dx}=\partial/\partial\dx$, whereas the upper dot is the total derivative with respect to $t$.

In Ref. \cite{Gorni2021} Gorni and Zampieri provided a theory to generate \textit{nonlocal constants} of motion for second order ODE Lagrangians using the concept of perturbed motions.  
Subsequently, in Ref. \cite{10077_31215}, the author gave a recipe to generate nonlocal constants for ODEs of any order. The paper \cite{GORNI2022100262} is a survey on the whole theory with a few sample applications.

\begin{theorem}\label{teo:nonloconst}
Let $t\rightarrow x(t)$ be a solution to the Euler-Lagrange equation \eqref{eq:E-L}, and let $x_\ve(t)$, $\ve\in\mathbb{R}$, be a smooth family of perturbed motions, such that $x_0(t)=x$. Then, the following function is constant:
\begin{equation}\label{eq:nonlocGZ}
I\equiv \partial_{\dx} L\big(t,x(t),\dx(t)\big)\cdot \partial_\ve x_\ve(t)\big\rvert_{\ve=0}-\int_{t_0}^t 
\partial_\ve L\big(s,x_\ve(s),\dx_\ve(s)\big)\big\rvert_{\ve=0}\,ds\,.
\end{equation}
\end{theorem} 
Expression \eqref{eq:nonlocGZ} means that the value of $I$ in $t$ depends not only on the state $(t,x(t),\dx(t))$, but also on the whole history of the motion from its beginning at $t_0$. Nevertheless, as described in Refs. \cite{10077_31215,Gorni2021}, one can find particular cases for which Theorem \ref{teo:nonloconst} yields true \emph{first integrals} for equation \eqref{eq:E-L}, in the sense of point functions $\mathcal{I}=\mathcal{I}(t,x(t),\dx(t))$ that are constant along solutions.

Generally, nonlocal constants and first integrals are precious tools to construct solutions of differential equations. In the case of linear ODE, there are a number of well defined methods to find nonlocal constants and first integrals. However, the same cannot be said for the nonlinear case, where the direct computation of constants is in general an open problem and only indirect techniques are available.

A general parametrization usually employed to study second-order nonlinear ODEs is the so-called \emph{Jacobi-type} form (or, \emph{Jacobi equation}). It is a second-degree polynomial in $\dx$ like this
\begin{equation}\label{eq:J}
\ddx+\tfrac{1}{2}\partial_x\vp\,\dx^2+\partial_t\vp \,\dx+B=0\,,
\end{equation}
where the coefficients $\vp=\vp(t,x)$ and $B=B(t,x)$ are analytical in $t$ and $x$.
By taking the parameters of Eq. \eqref{eq:J} appropriately, one can reproduce a great variety of nonlinear ODEs as a specific case.
For example, with \eqref{eq:J} one can easily recover many of the second-order equations of the \emph{Painlev\'{e}-Gambier} classification (see, e.g. Refs. \cite{10.1007/BF02419020, 10.1007/BF02393211,Dambrosi2021}).
These equations are  attracting much attention during last years, since some problems related to their solutions, the \emph{Painlev\'{e} transcendents}, are already under discussion  \cite{2019xdf}.

The background to the present work is as follows:
\begin{itemize}
\item In Ref.  \cite{GUHA20103247}, working directly on the ODE \eqref{eq:J}, Guha \emph{et al.}  employed  \emph{generalized Sundman transformations}  to derive a procedure to find constants for some of the \emph{Painlev\'{e}-Gambier} equations. In particular, we notice that part of their constants are nonlocal, due to the presence of an exponential factor whose argument is an integral of a function of $t$ and $x$ only.
More precisely, their nonlocal constants look like $f(t,\dx,x)\,e^{\int_t\!g(\tau,x(\tau)) d\tau}$.
\item In Ref.  \cite{Nucci2008UsingAO}, Nucci and Tamizhman presented a method devised by Jacobi to derive Lagrangians of any second-order ODE. Furthermore, among the examples provided, they demonstrated that the Jacobi equation \eqref{eq:J}  can be derived as the Euler-Lagrange equation \eqref{eq:E-L} for
$L(t,x,\dx)\equiv\tfrac{1}{2}e^{\vp}\dx^2+\delta_1\dx+\delta_2$
where $\delta_{1,2}=\delta_{1,2}(t,x)$ satisfy $\partial_t\delta_1-\partial_x\delta_2=e^\vp B$.
\end{itemize}

Motivated by the above, in this work we approach Theorem \ref{teo:nonloconst} as a new tool to provide explicit nonlocal constants and first integrals for the nonlinear Jacobi ODE \eqref{eq:J}. The paper is organized as follows. In Section \ref{sec:aut}, Theorem \ref{teo:xxx}  provides first integrals for the autonomous case  $\p{t}\vp=\p{t}B=0$. In Subsection \ref{eq:pt0}, Theorem \ref{teo:xxffx} provides nonlocal constants when $\p{t}\vp=0$ and $\p{t}B\ne 0$, whereas  in Subsection \ref{eq:lowe} our Theorem \ref{teo:xxkjkffx} provides first integrals and nonlocal constants when $\p{t}\vp\ne0$ and $\p{t}B\ne0$.
As far as we know, such theorems appear to be completely new to the literature. 
Similarly to Guha \emph{et al.}, our constants exhibit nonlocal contributes that depend on $t$ and $x$ only.

Finally, in Subsections \ref{sec:oiu}, \ref{sec:lkjhg}, and \ref{eq:koprwrt} we deal with some neat applications of our theorems. Among all, simple computations prove that our single formula \eqref{eq:r2} fully recovers the results obtained by Guha \emph{et al.} for the Painlev\'{e}-Gambier case.

Since nonlinear ODEs play a prominent role in many applied fields (computational science and engineering modeling, fluid dynamics, finance, and quantum mechanics, $...$), the main results of our work clearly indicates the powerful nature of nonlocal constants, and point out the necessity for a more extensive study of their applications.

\section{Autonomous Jacobi equation}
\label{sec:aut}

Given $\vp=\vp(x)$ and $B=B(x)$, in this section we study nonlocal constants and first integrals for the autonomous Jacobi equation
\begin{equation}\label{eq:reom}
\ddx+\tfrac{1}{2}\vp'\dx^2+B=0\,.
\end{equation}
Hereinafter, the prime symbol will stand for total derivatives with respect to $x$. 

In Ref.  \cite{Nucci2008UsingAO}, Nucci and Tamizhman present a method devised by Jacobi to find Lagrangians for any second-order ODE. In this respect, equation \eqref{eq:reom} 
can be derived as the Euler-Lagrange equation \eqref{eq:E-L} for
\begin{equation}\label{eq:rL}
\mathcal{L}(x,\dx)\equiv\tfrac{1}{2}e^{\vp}\dx^2+\delta_2\,,
\end{equation}
with $\delta_2=\delta_2(x)$ satisfying
\begin{equation}\label{eq:constrL}
\delta_2'=-e^{\vp}B\,.
\end{equation}

Let us consider expression \eqref{eq:nonlocGZ}, that evaluated on the  Lagrangian \eqref{eq:rL} gives
\begin{equation}\label{eq:nn}
I=e^\vp \dx \cdot \partial_\ve x_\ve\big\rvert_{\ve=0}
-\int_{t_0}^t \Big\{\big(\tfrac{1}{2}\vp' e^\vp\dx^2+\delta_2'\big)\cdot \partial_\ve x_\ve\big\rvert_{\ve=0}\,+e^\vp \dx\cdot \partial_\ve \dx_\ve\big\rvert_{\ve=0}\Big\}\, ds\,.
\end{equation}

In literature, some results for the time-independent Painlev\'{e}-Gambier case present constants which are nonlocal due to the presence of an integral in the argument of exponential factors (see, e.g. Ref. \cite{GUHA20103247}). Inspired by this fact, it is quite natural for us to consider as perturbed motion the $x$-shift family 
\begin{equation}\label{eq:pert}
x_\ve^\pm=x+\ve a\eb\,,
\end{equation}
where $a=a(x)$ and $b=b(x)$ are free functions. The $\pm$ formalism indicates that we are simultaneously considering both families with a positive and a negative sign in the argument of the exponential.

We point out that $\partial_\ve \dx_\ve\rvert_{\ve=0}=(\partial_\ve x_\ve\rvert_{\ve=0})^{\dt}$.  
Hence, using the family \eqref{eq:pert} and the constraint \eqref{eq:constrL}, expression \eqref{eq:nn} becomes
\begin{equation}\label{eq:I1}
I_\pm=a\dx e^{\vp}\eb-\int_{t_0}^t\!\Big\{\!\left(\tfrac{1}{2}a\vp'+a'\right)e^{\vp}\dx^2\pm a b e^{\vp}\dx-aBe^{\vp}\Big\}\,\eb ds\,.
\end{equation}

Formally, expression \eqref{eq:I1} exhibits a double integration that we would like to remove. In addition, as argued in the introduction, we want to avoid the dependence of integrands on derivatives of $x$. A possible approach in this direction is to reformulate  the integrand in $ds$ as a total derivative. In virtue of this, we neglect the term proportional to $\dx^2$ by imposing $\tfrac{1}{2}a\vp'+a'=0$, which gives
\begin{equation}\label{eq:a}
a=e^{-\vp/2}\,.
\end{equation}
Then, we  assume that $B\equiv b\B$ for some $\B=\B(x)$. 
Therefore, a direct computation with expression \eqref{eq:I1} gives
\begin{equation}\label{eq:I2}
I_\pm=a\dx e^{\vp}\eb-\int_{t_0}^t\!\Big\{\!(\pm a b \dx e^{\vp})\,\eb+(\mp a\B e^{\vp})\big(\eb\big)^{\dt}\Big\}\, ds\,.
\end{equation}
If we impose $\pm a b \dx e^{\vp}=(\mp a\B e^{\vp})^{\dt}$, which can be expanded as
\begin{equation}\label{eq:b}
b=-\tfrac{1}{2}\vp'\B-\B'\,,
\end{equation}
expression \eqref{eq:I2} becomes
\begin{equation}\label{eq:I3}
I_\pm=a\dx e^{\vp}\eb\pm\int_{t_0}^t\!\big(a\B e^{\vp}\eb\big)^{\dt}ds\,.
\end{equation}

We evaluate the integral in the second term of expression \eqref{eq:I3}. Hence, substituting expressions \eqref{eq:a} and \eqref{eq:b}, we finally get
\begin{equation}\label{eq:I4}
I_\pm=(\dx\pm\B)\,e^{\vp/2}e^{\mp\!\int_t(\vp'\!\B/2+\B')d\tau}\,.
\end{equation}

Note that expression \eqref{eq:I4} is a nonlocal constant for the autonomous Jacobi equation \eqref{eq:reom}, but it still depends on $\B$. 
By construction, equation \eqref{eq:b} constraints $\B$ as a function of $B$ and $\vp$. 
So, multiplying equation \eqref{eq:b} by $\B$ and defining $y\equiv\B^2$, we get the first-order linear ODE 
\begin{equation}\label{eq:lopr}
y'+\vp'y+2B=0\,.
\end{equation}
It is easy to see that equation \eqref{eq:lopr} admits $\B=\sqrt{y}$ as solution, with
\begin{equation}
\label{eq:rBbar}
\B=\sqrt{-2e^{-\vp}\smallint_xe^\vp B\,dx}
=\sqrt{2\delta_2 e^{-\vp}}\,.
\end{equation}
\begin{theorem}\label{teo:xxx}
Let $e^{\vp}B=-\delta_2'$ for some $\delta_2$. Then, $\mathcal{I}$ is a first integral for the autonomous Jacobi equation $\ddx+\tfrac{1}{2}\vp'\dx^2+B=0$ with
\begin{equation}\label{eq:I}
\mathcal{I}=\tfrac{1}{2}\dx^2e^{\vp}-\delta_2\,.
\end{equation}
\end{theorem} 
\begin{proof}
Use expression \eqref{eq:I4} to compute $\mathcal{I}\equiv\tfrac{1}{2} I_+I_-=\tfrac{1}{2}(\dx^2-\B^2)\,e^{\vp}$. Then, substitute expression \eqref{eq:rBbar} inside $\mathcal{I}$.
\end{proof}


\subsection{Application: autonomous Painlev\'{e}-Gambier equations}
\label{sec:oiu}

There are several equations of the \emph{Painlev\'{e}-Gambier} classification that belong to the autonomous Jacobi parametrization \eqref{eq:reom} (see, e.g. Refs. \cite{GUHA20103247,Dambrosi2021}).  For example, the Painlev\'{e}-Gambier equations {\small XVIII}, {\small XXI}, and {\small XXII} can be parametrized as
\begin{equation}\label{eq:Painleve}
\ddx-\tfrac{1}{2}\alpha x^{-1}\dx^2+\beta x^n=0\,,
\end{equation}
with some constant parameters $(\alpha,\beta,n)$, which are
\begin{eqnarray}\label{eq:coeffsPain}
(\alpha,\beta,n)_{\mbox{\tiny XVIII}}\!&=&\!\big(1,\!-4,2\big)\,,\nonumber\\
(\alpha,\beta,n)_{\mbox{\tiny XXI}}\!&=&\!\big(\tfrac{3}{2},\!-3,2\big)\,,\\
(\alpha,\beta,n)_{\mbox{\tiny XXII}}\!&=&\!\big(\tfrac{3}{2},1,0\big)\,.\nonumber
\end{eqnarray}

Let us first establish a comparison between equation \eqref{eq:Painleve} and equation \eqref{eq:reom} to infer that $\vp=-\alpha\ln x$ and $B=\beta x^n$. Then, from definition \eqref{eq:rBbar}, we get $\B=\sqrt{\beta \gamma}\, x^{(n+1)/2}$ with $\gamma\equiv 2/(\alpha-n-1)$. 

After direct computations,  expression \eqref{eq:I4} yields
\begin{equation}\label{eq:r1}
I_\pm=\big(\dx x^{-\alpha/2}\pm \sqrt{\beta \gamma}\, x^{-1/\gamma}\big)e^{\mp\sqrt{\beta/\gamma}\int_tx^{(n-1)/2}d\tau}\,,
\end{equation}
that is a nonlocal constant for Painlev\'{e}-Gambier parametrization \eqref{eq:Painleve}.

Similarly, Theorem \ref{teo:xxx} provides the first integral
\begin{equation}\label{eq:r2}
\mathcal{I}=\dx^2x^{-\alpha}-\beta \gamma x^{-2/\gamma}\,.
\end{equation}


When combined with coefficients \eqref{eq:coeffsPain}, our results \eqref{eq:r1} and \eqref{eq:r2} exactly recover from a non-local perspective the first integrals proposed case-by-case in Ref.  \cite{GUHA20103247}. 
We remark that the consistency of our machinery is fully confirmed even when applied to the whole \emph{Painlev\'{e}-Gambier} classification which falls into the autonomous Jacobi parametrization \eqref{eq:reom}.


\section{Non-autonomous Jacobi Equation}

Consider two reference parameters $\vp=\vp(t,x)$ and $B=B(t,x)$. In this section we study nonlocal constants and first integrals for the non-autonomous Jacobi equation
\begin{equation}\label{eq:eom}
\ddx+\tfrac{1}{2}\partial_x\vp\,\dx^2+\partial_t\vp \,\dx+B=0\,.
\end{equation}

In Ref. \cite{Nucci2008UsingAO} Nucci and Tamizhman derived equation \eqref{eq:eom} as the Euler-Lagrange equation \eqref{eq:E-L} for
\begin{equation}\label{eq:L}
L(t,x,\dx)\equiv\tfrac{1}{2}e^{\vp}\dx^2+\delta_1\dx+\delta_2\,,
\end{equation}
with $\delta_1=\delta_1(t,x)$ and $\delta_2=\delta_2(t,x)$ satisfying
\begin{equation}\label{eq:constrLL}
\partial_t\delta_1-\partial_x\delta_2=e^\vp B\,.
\end{equation}

Let us consider expression \eqref{eq:nonlocGZ}, that evaluated on the  Lagrangian \eqref{eq:L} gives
\begin{equation}\label{eq:xx}
I=\big(e^\vp \dx+\delta_1\big)\cdot \partial_\ve x_\ve\big\rvert_{\ve=0}
-\int_{t_0}^t \xi\, ds\,,
\end{equation}
where we defined $\xi=\xi(t,x)$ as
\begin{equation}\label{eq:XII}
\xi \equiv \big(\tfrac{1}{2}\p{x}\vp \,e^\vp\dx^2+\p{x}\delta_1\dx+\p{x}\delta_2\big)\cdot \partial_\ve x_\ve\big\rvert_{\ve=0}\,+\big(e^\vp \dx+\delta_1\big)\cdot \partial_\ve \dx_\ve\big\rvert_{\ve=0}\,.
\end{equation}

Taking inspiration from the autonomous case, we choose the $x$-shift family
\begin{equation}\label{eq:pertx}
x_\ve^\pm=x+\ve a\ebb\,,
\end{equation}
where $a=a(t,x)$ and $b=b_\pm(t,x)$ are  free functions. 

Notice that $\partial_\ve \dx_\ve\rvert_{\ve=0}=(\partial_\ve x_\ve\rvert_{\ve=0})^{\dt}$. Hence, applying  \eqref{eq:pertx} in expression \eqref{eq:XII} and using condition \eqref{eq:constrLL}, it follows that
\begin{multline}\label{eq:mkj}
\xi_\pm=
\big(a\delta_1\ebb\big)^{\dt}+
(\tfrac{1}{2}a\p{x}\vp+\p{x}a)e^\vp\dx^2\ebb\\
+\big\{\big[(\p{t}a\pm a b_\pm)\,e^\vp \dx\big]\ebb-(a e^\vp B)\,\ebb\big\}\,.
\end{multline}

Now, we want to transform $\xi_\pm$ into a total derivative.
For this purpose, the structure of expression \eqref{eq:mkj} suggests to begin neglecting the term proportional to $\dx^2$ by imposing $\tfrac{1}{2}a\p{x}\vp+\p{x}a=0$. Such condition is solved by
\begin{equation}\label{eq:at}
a=e^{-\vp/2}\,.
\end{equation}
Some algebraic manipulations rewrite expression \eqref{eq:mkj} as
\begin{equation}\label{eq:kll}
\xi_\pm\!=\!
\big(a\delta_1\ebb\big)^{\dt}
\!+\!\big\{\big[(\p{t}a\pm a b_\pm)\,e^\vp \dx\big]\ebb\!+\!(\mp a e^\vp \B_\pm)\,\big(\ebb\big)^{\dt}\big\}.
\end{equation}
where we assumed $B\equiv b_\pm\B_\pm$ for some $\B_\pm=\B_\pm(t,x)$. 
We impose that the free parameter $b_\pm$ satisfies the following constraint
\begin{equation}\label{eq:dd}
(\p{t}a\pm a b_\pm)\,e^\vp \dx=(\mp a e^\vp \B_\pm)^{\dt}\,,
\end{equation}
that finally transforms $\xi_\pm$ into a total derivative
\begin{equation}\label{eq:zxc}
\xi_\pm=\big[a\delta_1\ebb\mp(ae^\vp\B_\pm)\ebb\big]^{\dt}\,.
\end{equation}

Furthermore, using \eqref{eq:pertx}, it follows by \eqref{eq:zxc} that expression \eqref{eq:xx} is equivalent to
\begin{equation}\label{eq:ddd}
I_\pm=(\dx\pm\B_\pm)\,a e^\vp\ebb\,,
\end{equation}
This result depends on the $b_\pm$ and $\B_\pm$ parameters, which are unknown so far.

Consider condition \eqref{eq:dd} and expand it as follows
\begin{equation}\label{eq:lklk}
\mp(\p{t}a\pm a b_\pm)e^\vp \dx = \p{t}(a e^\vp \B_\pm)+\p{x}(a e^\vp \B_\pm)\dx\,.
\end{equation} 
Equation \eqref{eq:lklk} can be decomposed as
\begin{equation}
\label{eq:klk}
\begin{cases}
\p{t}(a e^\vp \B_\pm)=0\,,\\
\p{x}(a e^\vp \B_\pm)=\mp(\p{t}a\pm a b_\pm)e^\vp\,.
\end{cases}
\end{equation}
It is immediate to check that the first equation of \eqref{eq:klk} gives 
\begin{equation}\label{eq:k}
k_\pm\equiv a e^\vp \B_\pm\,,
\end{equation}
with $k_\pm=k_\pm(x)$ a time-independent function. 

On the other hand, using the mixed partials equality $\p{x}\p{t}k_\pm=\p{t}\p{x}k_\pm$, the second equation of \eqref{eq:klk} provides the key relation
\begin{equation}
\p{t}\big[(\p{t}a\pm a b_\pm)e^\vp\big]=0\,,
\end{equation}
that can be expanded as
\begin{equation}\label{eq:ssdf}
a\,\p{t}b_\pm+(\p{t}\vp\, a+\p{t}a)\,b_\pm\pm(\p{t}\vp\p{t}a+\p{t}\p{t}a)=0\,.
\end{equation}

We note that \eqref{eq:ssdf} is a first-order (linear) Partial Differential Equation (PDE) for $b_\pm$.  We solve equation \eqref{eq:ssdf} using standard techniques and impose condition \eqref{eq:klk} to fix the integration constant. So, after several computations, we get
\begin{equation}\label{eq:fgh}
b_\pm=\mp\,\p{t}a-e^{-\vp}a^{-1}\p{x}k_\pm\,.
\end{equation}

Since $B\equiv b_\pm\B_\pm$, we multiply expression \eqref{eq:fgh} by $\B$ and use definition \eqref{eq:k} to substitute $k_\pm$. Consequently, expression simplifies to
\begin{equation}\label{eq:kjl}
\B_\pm\p{x}\B_\pm+\p{x}(\ln a+\vp)\B_\pm^2\pm(\p{t}\ln a) \B_\pm=-B\,,
\end{equation}
that finally constraints $\B_\pm$ in terms of the reference parameters $B$ and $\vp$.

In contrast with the autonomous case, equation \eqref{eq:kjl} is a hard-to-solve nonlinear PDE for $\B_\pm$. 
We choose to proceed by splitting our study depending on wether $\p{t}\vp$ (more precisely $\p{t}\ln a=-\tfrac{1}{2}\p{t}\vp$) is equal to zero or not.
  
  
\subsection{Case $\p{t}\vp=0$}
\label{eq:pt0}

If $\p{t}\vp=0$, equation \eqref{eq:kjl} becomes
\begin{equation}\label{eq:ghj}
\B_\pm\p{x}\B_\pm+\p{x}(\ln a+\vp)\B_\pm^2=-B\,.
\end{equation}
As in the autonomous case, substituting definition \eqref{eq:at} in equation \eqref{eq:ghj} and assuming $y_\pm\equiv\B_\pm^2$, we get the first-order linear PDE 
\begin{equation}
\p{x} y_\pm+(\p{x}\vp) y_\pm+2B=0\,, 
\end{equation}
which admits $\B\equiv\B_\pm=\sqrt{y_\pm}$ as solution, with
\begin{equation}\label{eq:rBbare}
\B=\sqrt{-2e^{-\vp}\smallint^x_{x_0}(e^\vp B)\rvert_{t}\,dx}=
\sqrt{2e^{-\vp}\smallint^x_{x_0}(\p{x}\delta_2-\p{t}\delta_1)\rvert_{t}\,dx}\,.
\end{equation}

In a similar way, being $\p{t}\vp=0$, equation \eqref{eq:fgh} can be rewritten as
\begin{equation}\label{eq:lkiqw}
b_\pm=e^{-\vp}a^{-1}\p{x}k_\pm\,.
\end{equation}
Hence, applying expressions \eqref{eq:rBbare} and \eqref{eq:k} in \eqref{eq:lkiqw}, it turns out that
\begin{equation}\label{eq:lkop}
b\equiv b_\pm=-\tfrac{1}{2}\p{x}\vp-\p{x}\B\,.
\end{equation}

Therefore, expressions \eqref{eq:rBbare} and \eqref{eq:lkop} can be used to evaluate the nonlocal constant \eqref{eq:ddd}, that becomes
\begin{equation}\label{eq:ddddsd}
I_\pm=(\dx\pm\B)\, e^{\vp/2}e^{\mp\int_t(\p{x}\vp/2+\p{x}\B)d\tau}\,.
\end{equation}

\begin{theorem}\label{teo:xxffx}
Let $\p{t}\vp=0$ and $e^\vp B=\p{x}(\partial_t\eta-\delta_2)$ for some $\eta$ and $\delta_2$. Then, $\mathcal{I}$ is  constant along the solutions of the Jacobi equation $\ddx+\tfrac{1}{2}\partial_x\vp\,\dx^2+B=0$ with
\begin{equation}\label{eq:Il}
\mathcal{I}\equiv\tfrac{1}{2}\dx^2e^{\vp}+\partial_t\eta-\delta_2
-\smallint^t_{t_0}\p{t}(\partial_t\eta-\delta_2)\,dt
\end{equation}
\end{theorem} 
\begin{proof}
Use expression \eqref{eq:ddddsd} to compute $\mathcal{I}\equiv\tfrac{1}{2} I_+I_-=\tfrac{1}{2}(\dx^2-\B^2)\,e^{\vp}$. Define $\delta_1\equiv \p{x}\eta$ and consider expression \eqref{eq:rBbare}, that can be rewritten as 
\begin{eqnarray}
\B^2&=&2e^{-\vp}\smallint^x_{x_0}\p{x}\big(\delta_2-\p{t}\eta\big)\rvert_{t}\,dx\,,\\
\label{eq:lol}
&=&2e^{-\vp}\big[\delta_2-\p{t}\eta-\smallint^t_{t_0}\p{t}\big(\delta_2-\p{t}\eta\big)\,dt\big]\,.
\end{eqnarray}
Finally, substitute expression \eqref{eq:lol} inside $\mathcal{I}$.
\end{proof}
Notice that, as expected, the integrand of \eqref{eq:Il} depends only on $x$ and $t$. In addition, when evaluated in the autonomous case, expression \eqref{eq:Il} exactly recovers our previous result \eqref{eq:I}.

\begin{remark}
One might wonder under what conditions expression \eqref{eq:Il} becomes a true first integral, in the sense of a local function of $t$. For this purpose, it is sufficient to require that $\p{t}(\partial_t\eta-\delta_2)=\dot\psi$ for some $\psi=\psi(t)$. Derive such condition with respect to $x$, switch partial derivatives, and use the hypothesis $e^\vp B=\p{x}(\partial_t\eta-\delta_2)$ of Theorem \ref{teo:xxffx} to finally get
$0=\p{t}(e^\vp B)=e^\vp \p{t}B$, which implies $\p{t}B=0$. Since in this subsection we also assume $\p{t}\vp=0$, expression \eqref{eq:Il} returns a true first integral only if we move to the autonomous case.
\end{remark}


\subsubsection{Application: non-autonomous Painlev\'{e}-Gambier equations}
\label{sec:lkjhg}

There are several equations of the \emph{Painlev\'{e}-Gambier} classification which belong to the Jacobi parametrization \eqref{eq:eom} (see, e.g. Refs. \cite{GUHA20103247,Dambrosi2021}).  
For example, the Painlev\'{e}-Gambier equation {\small IV} is
\begin{equation}\label{eq:Painlevesd}
\ddx-(6x^2+t)=0\,.
\end{equation}

Let us first establish a comparison between equation \eqref{eq:Painlevesd} and Theorem \ref{teo:xxffx} to infer that $\vp=0$ and $B=-(6x^2+t)$. Since $e^\vp B=-(6x^2+t)$, we can chose $\eta=-2x^3t$ and $\delta_2=tx$. Then, expression \eqref{eq:Il} provides the nonlocal  constant
\begin{equation}
\mathcal{I}=\tfrac{1}{2}\dx^2-2x^3-xt+\smallint_{t_0}^t x\,dt\,.
\end{equation}

To give an another example, the Painlev\'{e}-Gambier equation {\small XX} is
\begin{equation}\label{eq:Painlevesgfgfd}
\ddx-\tfrac{1}{2}x^{-1}\dx^2-4x^2-2tx=0\,.
\end{equation}

A comparison between equation \eqref{eq:Painlevesgfgfd} and Theorem \ref{teo:xxffx} yields $\vp=-\ln x$ and $B=-2x(2x+t)$. Since $e^\vp B=-2(2x+t)$, we can chose $\eta=-t^2x$ and $\delta_2=2x^2$. Then, expression \eqref{eq:Il} yields
\begin{equation}
\mathcal{I}=\tfrac{1}{2}\dx^2x^{-1}-2x(x+t)-2\smallint_{t_0}^t x\,dt\,.
\end{equation}
We remark that the consistency of our Theorem \ref{teo:xxffx} is also fully confirmed when applied to the whole \emph{Painlev\'{e}-Gambier} classification which falls into the Jacobi parametrization \eqref{eq:eom} with $\p{t}\vp=0$. 


\subsection{Case $\p{t}\vp\ne0$}
\label{eq:lowe}

If $\p{t}\vp\ne0$, equation \eqref{eq:kjl} is a non-linear PDE that doesn't admit an explicit solution.
Some manipulations are therefore necessary to proceed.

Multiply equation \eqref{eq:kjl} by $e^{2(\ln a+\vp)}$ and apply definition \eqref{eq:k} to get
\begin{equation}\label{eq:lki}
\p{x}k_\pm^2\pm2\p{t}a e^\vp k_\pm=-2 B a^2 e^{2\vp}\,.
\end{equation}
Take the time derivative $\partial_t$ of equation \eqref{eq:lki} and switch $\p{t}\p{x}k^2_\pm=\p{x}\p{t}k^2_\pm$
\begin{equation}\label{eq:lopw}
\p{x}\p{t}k^2_\pm\pm2\p{t}(\p{t}a e^\vp) k_\pm\pm2\p{t}a e^\vp (\p{t} k_\pm)=-2\p{t}(B a^2 e^{2\vp})\,.
\end{equation}
Furthermore, since $\p{t}k^2_\pm=2k_\pm\p{t}k_\pm=0$ (remember that expression \eqref{eq:klk} fixes $\p{t}k_\pm=0$),  from equation \eqref{eq:lopw} it follows that 
\begin{equation}\label{eq:swe}
\p{t}(\p{t}a\, e^\vp) k_\pm=\mp\p{t}(B a^2 e^{2\vp})\,,
\end{equation}
which can be simply solved with respect to $k_\mp$, obtaining therefore
\begin{equation}\label{eq:klo}
k_\pm=\mp\frac{\p{t}(B a^2 e^{2\vp})}{\p{t}(\p{t}a\, e^\vp)}\,.
\end{equation}
We remember that equation \eqref{eq:klo} must satisfy conditions in \eqref{eq:klk}.

By definition \eqref{eq:k}, $\B_\pm= a^{-1} e^{-\vp}k_\pm$. Hence, replacing $k_\pm$ and $a$ in the right hand side of \eqref{eq:klo} with expressions \eqref{eq:klo} and \eqref{eq:at} respectively, we finally obtain
\begin{equation}\label{eq:lopopo}
\B_\pm=\pm4\frac{\p{t}B+B\p{t}\vp}{2\p{t}\p{t}\vp+(\p{t}\vp)^2}\,.
\end{equation}
On the other hand,  we can use expression \eqref{eq:lopopo} to calculate $b_\pm=B/\B_\pm$. Hence,
\begin{equation}\label{eq:lnddb}
b_\pm=\pm\frac{1}{4}\frac{2\p{t}\p{t}\vp+(\p{t}\vp)^2}{\p{t}\ln B+\p{t}\vp}\,.
\end{equation}
This leads to our new result for this section.
\begin{theorem}\label{teo:xxkjkffx}
Let $\p{t}\vp\ne0$ and $Be^\vp=\rho_1 e^{\vp/2}+\rho_2$ for some $\rho_{1,2}=\rho_{1,2}(x)$ such that
\begin{equation}\label{eq:rr}
\rho_1'=\p{t}(\ln\p{t}\vp)\big(e^{\vp/2}+\rho_2/\rho_1\big)\,.
\end{equation}
Then, $I$ is constant along the solutions of the Jacobi equation $\ddx+\tfrac{1}{2}\partial_x\vp\,\dx^2+\partial_t\vp \,\dx+B=0$ with
\begin{equation}\label{eq:lkj}
\!I\!=\!\left(\!\dx e^{\vp/2}\!+\!\frac{2\rho_1}{\p{t}(2\ln \p{t}\vp\!+\!\vp)}\!\right)\exp\left\{\frac{1}{2}\!\int_t 
\p{t}(2\ln \p{t}\vp\!+\!\vp)\!\left(\!1\!+\!\frac{\rho_2}{\rho_1}e^{-\vp/2}\!\right)\!d\tau\right\}
\end{equation}
\end{theorem}
\begin{proof}
Substitute \eqref{eq:at} in expression \eqref{eq:klo} and impose $\p{t}{k_\pm}=0$ (see the first condition of \eqref{eq:klk}). 
The obtained equation is $\p{t}{(B e^\vp})=\rho_1\p{t}e^{\vp/2}$, which is solved by $B=\rho_1 e^{-\vp/2}+\rho_2 e^{-\vp}$ (our first hypothesis). Compute $B$ inside expressions \eqref{eq:lopopo} and \eqref{eq:lnddb}. Use the obtained expressions for $\B_\pm$ and $b_\pm$ inside definition \eqref{eq:ddd} to get our final result
\eqref{eq:lkj}. Note that in this case $I_+=I_-$, hence $I=I_\pm$. Since $k_\pm=\pm\rho_1$, our second hypothesis \eqref{eq:rr} is obtained from $\p{x}k_\pm=0$ (the second condition of \eqref{eq:klk}).
\end{proof}
Notice that, as expected, the integrand of \eqref{eq:lkj} depends only on $x$ and $t$ 

\begin{remark}\label{rem:df}
One might wonder under what conditions expression \eqref{eq:lkj} becomes a true first integral, in the sense of a local function of $t$. For this purpose, since the integrand depends only on $t$ and $x$, it is sufficient to require that $\p{t}(2\ln \p{t}\vp\!+\!\vp)\!\left(1\!+\!\rho_2/\rho_1e^{-\vp/2}\right)=\dot F(t)$ for some $F=F(t)$.
\end{remark}


\subsubsection{Application: non-autonomous Jacobi equation}
\label{eq:koprwrt}

Consider the following Jacobi-type equation
\begin{equation}\label{eq:jaerctot}
\ddx+\tfrac{1}{2}\dx^2+\dx+\varrho\, e^{-(t+x)/2}=0\,,
\end{equation}
with $\varrho$ a constant parameter. Since $B=\varrho\, e^{-(t+x)/2}$ and $\p{x}\vp=\p{t}\vp=1$, we have $\vp=t+x$. Hence, equation \eqref{eq:jaerctot} satisfies the hypotheses of Theorem \ref{teo:xxkjkffx} with $\rho_1=\varrho$ and $\rho_2=0$.  Note that Remark \ref{rem:df} is also satisfied by $F=t$. Hence, expression \eqref{eq:lkj} yields the following true first integral
\begin{equation}\label{eq:lpo}
\mathcal{I}=e^{\,t/2}\left(\dx e^{(t+x)/2}+2\varrho\right).
\end{equation}

Being linear in $\dx$, $\mathcal{I}$ turns out to be a precious tool to find a solution of equation \eqref{eq:jaerctot}. In fact, expression \eqref{eq:lpo} can be rewritten as $\dot{\mathcal{J}}(t,x)=0$, with
\begin{equation}\label{eq:kloi}
\mathcal{J}= 2 \,e^{\,x/2}+\tilde{\mathcal{I}} e^{-t}-4\varrho \,e^{-t/2}\,.
\end{equation}
Hence, $\mathcal{J}$ is a first integral too. After some algebraic manipulations, it can be easily checked that expression \eqref{eq:kloi} solves directly equation \eqref{eq:jaerctot} with
\begin{equation}
x(t)=2\ln\left\{2\varrho \,e^{-t/2}-\tfrac{1}{2}\,\tilde{\mathcal{I}}\, e^{-t}+\tilde{\mathcal{J}}\right\}\,.
\end{equation}
Here, the $\tilde{\mathcal{I}}$ and $\tilde{\mathcal{J}}$ parameters are constants.


\section{Acknowledgments}
\label{sec:ak}
The author would like to thank Professor Gaetano Zampieri for useful discussions.

\nocite{*}
\providecommand{\bysame}{\leavevmode\hbox to3em{\hrulefill}\thinspace}


\end{document}